\def\FF{{\mathbb F}}
\def\fn{{\mathbb F}_{2^n}}
\def\f2{{\mathbb F}_{2}}

\def\cA{\mathcal A}
\def\cB{\mathcal B}
\def\cC{\mathcal C}

\def\cK{\mathcal K}
\def\cL{\mathcal L}

\def\l{\ell}

\newcommand{\g}{\gamma}
\newcommand{\gd}{\delta}

\newcommand{\gl}{\lambda}

\newcommand{\gr}{\rho}
\newcommand{\gs}{\sigma}

\newcommand{\gt}{\tau}

\newcommand{\gG}{\Gamma}

\def\Im{\mathrm{Im}}
\newcommand{\Sym}{{\rm{Sym}}}
\newcommand{\GL}{{\rm{GL}}}

\documentclass{ws-jaa}
\newtheorem{theore}{Theorem}[section]
\newtheorem{lemm}[theore]{Lemma}
\newtheorem{remar}[theore]{Remark}
\newtheorem{propositio}[theore]{Proposition}
\newtheorem{definitio}[theore]{Definition}
\begin{document}

\markboth{M. Calderini}
{Primitivity of the group of a cipher involving the action of the key-schedule}

%
\catchline{}{}{}{}{}
%

\title{Primitivity of the group of a cipher involving the action of the key-schedule
}

\author{ Marco Calderini
}

\address{Department of Informatics, University of Bergen\\ Bergen, NO\\
\email{marco.calderini@uib.no}}

\maketitle

\begin{history}
\received{(Day Month Year)}
\revised{(Day Month Year)}
\accepted{(Day Month Year)}
\comby{(xxxxxxxxx)}
\end{history}

\begin{abstract}
The algebraic structure of the group generated by the encryption functions of a block cipher depends on the key schedule algorithm used for generating the round keys. For such a reason, in general, 
studying this group does not appear to be an easy task.
Previous works, focusing on the algebraic properties of groups associated to a cipher, have studied the group generated by the round functions of the cipher considering independent round keys.

In this paper, we want to study the more realistic group generated by the encryption functions, where the key schedule satisfies certain requirements. 
In this contest, we are able to identify sufficient conditions that permit to guarantee the primitivity of this group and the security of the cipher with respect to the partition-based trapdoor. This type of trapdoor has been recently introduced by Bannier et al. (2016) and it is a generalization of that introduced by Paterson in 1999.
\end{abstract}


\ccode{2000 Mathematics Subject Classification: 20B15, 20B35, 94A60}

\section{Introduction}

Since the work \cite{coppersmith1975generators}, where Coppersmith and Grossman considered a set of functions, to define a block cipher and studied the permutation group
 generated by these, much attention has been devoted to the group generated by the round functions of a block cipher (see for instance \cite{calderinisalerno,CGC-cry-art-carantisalaImp,even,Wernsdorf2}). Investigating the group theoretic structure of a block cipher is important for identifying and excluding undesirable properties. In this context, Paterson \cite{CGC-cry-art-paterson1} showed that the imprimitivity of the group generated by the round functions of a block cipher can represent a weakness exploitable for constructing a trapdoor. By a trapdoor it is meant an algebraic structure, hidden in the cipher design, that would allow an attacker to break the cipher easily. 
 
Most modern block ciphers belong to two families of symmetric cryptosystems, that is, Substitution-Permutation Networks (SPN) and Feistel Networks (FN). In this paper, we focus on iterated block ciphers of type SPN using the XOR operation for key addition. In \cite{CGC-cry-art-carantisalaImp}, these types of ciphers have been also called {\em translation-based} (tb) ciphers. For a tb cipher $\cC$, in \cite{calderinisalerno,aragona2014,CGC-cry-art-carantisalaImp}, the authors study the group generated by the round functions of $\cC$, $\gG_\infty(\cC)$, providing cryptographic conditions on the S-boxes and the mixing layer which guarantee the primitivity of the group.
In a recent work \cite{bannier2016partition} (see also \cite{bannier2017partition}), the authors introduce and study the partition-based trapdoor. This algebraic trapdoor generalizes that introduced by Paterson.

In \cite{calderini2018} the author shows that studying the group generated by the round functions could be not enough to guarantee the security of the cipher with respect to the imprimitivity and/or partition based trapdoor (see \cite[Section 4]{calderini2018}). For such a reason, in \cite{calderini2018} it has been studied a subgroup of the group generated by the round functions, namely the group generated by the encryption functions with independent round keys $\gG_{ind}(\cC)$. In \cite{calderini2018} cryptographic properties on the S-boxes of the block cipher, sufficient to guarantee the security of such a group to these trapdoors, are given. 

From the results of these previous works, we have which properties of the components of the round functions of the cipher can guarantee that the
trapdoors cannot work for all encryption functions, when we have independent round-keys uniformly distributed.

In this work, we focus on the group $\gG(\cC)$ generated by the encryption functions of a cipher $\cC$. Note that such a group depends on the key-schedule algorithm, and it is contained in the group generated by the encryption functions with independent round-keys. 
We are interested in investigating such a group since as for the case of the group generated by the round functions, it could be that also the security of the group studied in \cite{calderini2018} cannot guarantee the security of the cipher if the key-schedule does not generate independent round-keys.

Since $\Gamma(\cC)$ is the smallest group containing the encryption functions of $\cC$, it is important to understand its properties. However, studying such a group it is not easy, in general, since it strictly depends on the {round-keys} generated by the key-schedule algorithm.
In this work, we want to investigate some properties of the key-schedule and of the components of the round functions of the cipher that permit to determine the security of the cipher $\cC$ with respect to the algebraic attack of the partition-based trapdoor and, thus, the primitivity of the group $\Gamma(\cC)$.

The paper is organized as follows. In Section \ref{sec:pre}, we recall some definitions and give some results presented in \cite{bannier2016partition,calderini2018}. In Section \ref{sec:gamma}, we report our study on the group $\gG(\cC)$. In particular, in the first part of Section \ref{sec:gamma}, we give a first property for the key-schedule, based on having independent round keys on three rounds. This permits to use the cryptographic properties given in \cite{calderini2018} also for avoiding the partition trapdoor and obtain the primitivity of $\gG(\cC)$.
In the second part, we consider a subgroup $T$ of order $2^{n-1}$ of the translation group and analyze the possible partitions which can be mapped one into the other by $T$. From such analysis we are able to provide a weaker property for the key-schedule algorithm, for which is possible to identify cryptographic properties for obtaining the primitivity for the group $\gG(\cC)$.

\section{Preliminaries and notation}\label{sec:pre}

Let $\cC$ be a block cipher acting on a message space $V = (\FF_2)^n$, for some $n\ge 1$. Let $\cK$ be the  keys space. Then our cipher $\cC$ is given by a family of key-dependent permutations $\gt_k$, called encryption functions, 
$$\cC=\{\gt_k\mid k \in\cK\}.$$

We are interested in studying the algebraic properties of the group
$\Gamma(\cC)=\langle \gt_k\mid k \in\cK\rangle$. The study of $\Gamma(\cC)$ is a difficult task, in general. 

For the well-known cipher DES, there are some works analyzing the group $\gG(DES)$ (see for instance \cite{camp,cop,Kaliski}). In these works, using some weak-keys generated by the key-schedule of DES, it is provided a lower bound on the order of this group. These lower bounds permit to prove also that 3-DES is different from DES. 
Moreover, in \cite{Kaliski}, the authors provide some evidences suggesting that the group generated by DES is as large as the alternating group, showing also that if the permutation group generated by the encryption functions of a cipher is too small, then the cipher is vulnerable to birthday-paradox attacks. Note that the fact that the group $\Gamma(\cC)$ is large cannot be a security measure standing on its own \cite{MKPW}.

To the best of our knowledge, these works,  focusing on the cipher DES, are the only ones studying the properties of $\gG(\cC)$ for a given block cipher $\cC$.

\subsection{Translation based ciphers}

Most modern block ciphers are obtained by a composition of several key-dependent permutations, called round functions. That is, any encryption function $\gt_k$, of a cipher $\cC$, is a composition
of some permutations  $\gt_{k,1}, ... , \gt_{k,\ell}$ of $V$. Here we consider translation-based cipher, introduced in \cite{CGC-cry-art-carantisalaImp}. 

Let $m, b > 1$ and
$$
V=V_1\oplus\dots\oplus V_b,
$$
where the spaces $V_i$'s are isomorphic to $(\FF_2)^m$. 

We will denote the group of all permutations on $V$ by $\Sym(V)$. Given $v \in V$, we denote by $\gs_v\in \Sym(V )$ the translation of $V$ mapping $x$ to $x + v$. The group of all translations of $V$ will be denoted by $T(V)$. In the following, the action of $g \in \Sym(V)$ on an element $v\in V$ will be written as $vg$.

For any element $v\in V$, we can write $v=v_1\oplus\dots \oplus v_b$, for some $v_i\in V_i$. The map $\pi_i:V\to V_i$ is the projection mapping $v\mapsto v_i$.

A permutation $\gamma\in \Sym(V)$ acting as $v \gamma=v_1\gamma_1\oplus\dots\oplus v_b\gamma_b$, for all $v\in V$ and for some fixed $\gamma_i\in \Sym(V_i)$, is called a {\em parallel map} (or ``parallel S-box'') and any $\gamma_i$ is called an {\em S-box}. In the cryptographic context, a linear map $\gl:V\to V$ is traditionally called a ``Mixing Layer'' when used in composition with parallel maps, in the round functions of a cipher. 

For any $I\subset \{1,...,b\}$, with $I\ne \emptyset$ and $I\ne \{1,...,b\}$, we define $\bigoplus_{i\in I} V_i$ a {\em wall}. 
\begin{definitio}
A linear map $\gl\in \GL(V)$ is called a {\em proper mixing layer} if no wall is invariant under $\gl$. Moreover, we say that $\gl$ is {\em strongly-proper} if no wall is mapped in another wall.
\end{definitio}

We can characterize translation-based block ciphers by the following:
\begin{definitio}[\cite{CGC-cry-art-carantisalaImp}]\label{def:tb}
 A block cipher  $\mathcal{C} = \{ \gt_k \mid k  \in \mathcal{K} \}$ over
  $V=({\FF_2})^{n}$ is called {\em translation-based (tb)} if:
       \begin{itemize}
    \item[1)] it is the composition of a finite number of rounds, $\ell$, such that any round $\gt_{k,h}$ is given by the composition of three maps $\gamma_h,\gl_h$ and $\gs_{ k_h}$, where
    \begin{itemize}
    \item[-] $\gamma_h$ is a parallel S-box (depending on the round but not on $k$) and $0\gamma_h=0$\footnote{ The assumption $0\gamma_h = 0$ is not restrictive. Indeed, we can always include $0\gamma_h$ in the round key addition of the previous round (see \cite[Remark 3.3]{CGC-cry-art-carantisalaImp}). 
},
     \item[-] $\lambda_h$ is a bijective linear map (depending on the round but not on $k$),
    \item[-] $\gs_{k_h}$ is the translation by the round key (it depends on both $k$ and the round),
     \end{itemize}
     
      \item[2)] for at least one round, we have (at the same time) that $\gl_h$ is proper and that the map $\Phi_h:\mathcal{K} \to V$ given by $k \mapsto k_h$ is
      surjective.
    \end{itemize}
\end{definitio}

The round-keys are determined by a {\em key-schedule} algorithm, which is a map $\Phi:\mathcal{K}\to V^\ell$, $k\mapsto (k_1,...,k_\ell)$. The map $\Phi_h$ in Definition~\ref{def:tb} is given by $\pi_h\circ \Phi$.

The tb characterization allowed the authors of \cite{CGC-cry-art-carantisalaImp} to investigate the permutation group generated by the round functions of $\cC$. In particular, for any round $h$, let 
$$\Gamma_h(\cC)=\langle \gt_{k,h}\mid k \in\cK \rangle,$$
therefore, we can define the group containing $\Gamma(\cC)$ generated by the round functions
$$\Gamma_\infty(\cC)=\langle \Gamma_h(\cC)\mid h=1,...,\ell \rangle.$$
It is easy to show, see for instance \cite[Lemma 2.2]{calderinisalerno}, that the functions involved in the round functions of a tb cipher defined over a binary vector space are even, i.e. $\Gamma_\infty(\cC)$ is contained in the alternating group  $\mathrm{Alt}(V)$. Exploiting the O'Nan-Scott theory of primitive groups, it is possible to prove that several relevant ciphers generate a large group $\Gamma_\infty(\cC)$ (see for instance \cite{calderinisalerno,aragona2014,caranti}).

As shown in \cite{calderini2018}, studying the group generated by the round functions could be not enough to guarantee the security of the cipher or of the group $\Gamma(\cC)$. For this reason the author studied the group

$$\gG_{ind}(\cC)=\langle \gt_K \mid K\in V^\ell \rangle,$$
where, letting $K = (k_1 , . . . , k_\ell )$, $\gt_K$ is the encryption function obtained using $k_h\in V$
as round-key at round $h$ for any $1\le h\le \ell$. Clearly, the group $\gG_{ind}(\cC)$ is such that $$\gG(\cC)\subseteq \gG_{ind}(\cC)\subseteq\Gamma_\infty(\cC).$$

In the following we will consider as tb ciphers all ciphers satisfying just the first condition in Definition~\ref{def:tb}.

\subsection{Boolean functions}

Let $m\ge 1$, and let $f:(\FF_2)^m\rightarrow(\FF_2)^m$ be a vectorial Boolean function. The derivative of $f$ in the direction $u\in(\FF_2)^m$ is given by $\hat{f}_u(x)=f(x+u)+f(x)$. 
\begin{definitio}
Let $f$ be a {vectorial Boolean function}, define for any $a,b\in{(\FF_2)^m}$
$$
\gd_f(a,b)=|\{x\in{(\FF_2)^m} \mid\, \hat{f}_a(x)=b\}|.
$$
Then, $f$ is said to be {\em differentially $\gd$-uniform} if 
$$
{\gd}=\max_{a,b \in{(\FF_2)^m},\atop
a\neq 0}\gd_f(a,b)\,.
$$
\end{definitio}

Vectorial Boolean functions, having a low differential uniformity, may prevent differential cryptanalysis (introduced by Biham and Shamir \cite{CGC2-cry-art-biham1991differential}) when used as S-boxes in block ciphers. From this point of view, functions having differential $2$-uniformity (the smallest possible in even characteristic) are optimal. Functions differentially $2$-uniform are called {\em almost perfect nonlinear} (APN).

 By \cite[Fact 3]{CGC-cry-art-carantisalaImp}, a vectorial Boolean function differentially $\gd$-uniform satisfies
$$|{\rm Im}(\hat{f}_u)|\geq\frac{2^{m}}{\delta}.$$
Moreover, in a similar way, we can see that for any subset $S$ of $(\FF_2)^m$, $$|\hat{f}_u(S)|\geq\frac{|S|}{\delta}.$$

The following notion, introduced in \cite{CGC-cry-art-carantisalaImp}, focus on the image of vector spaces.
\begin{definitio}
Let $1\leq r<m$ and $f(0)=0$, we say that $f$ is {\it strongly $r$-anti-invariant} if, for any two subspaces $U$ and $W$ of $(\mathbb F_2)^m$ such that $f(U)=W$, then either $\dim(U)=\dim(W)<m-r$ or $U =W=(\mathbb F_2)^m$.
\end{definitio}

\subsection{Partition-based trapdoors}

Let $G\subseteq \Sym(V)$ be a transitive permutation subgroup. $G$ is called {\em primitive} if it has no nontrivial $G$-invariant partition of $V$.
That is, there does not exist a nontrivial partition $\cA\ne\{\{v\}\mid v \in V\}, \{V\}$ of $V$, such that $Ag\in \cA$ for all $A\in\cA$ and $g\in G$.
 On the other hand, if a nontrivial $G$-invariant partition exists, the group is called {\em imprimitive}.

As mentioned above, imprimitivity is an undesirable property of the group $\Gamma_{\infty}(\cC)$. Indeed, Paterson in \cite{CGC-cry-art-paterson1} constructs a DES-like cipher whose group generated by round functions is imprimitive. Using this property he is able to mount an attack on the cipher.

From the idea of the imprimitive action trapdoor, in a recent work \cite{bannier2016partition}, Bannier et al. introduce partition-based trapdoors. In this work, the authors show how to construct a tb cipher such that there is a (non-trivial) partition of the message space $V$ which is mapped onto another partition by the encryption functions.

In the following, we report some definitions and results from \cite{bannier2016partition}.

\begin{definitio}
Let $\gr$ be a permutation of $V$ and $\cA, \cB$ be two partitions of $V$. Let $\cA \gr=\{A\gr \mid A \in \cA\}$. We say that {\em $\gr$ maps $\cA$ onto $\cB$} if and only if $\cA\gr = \cB$. 

Moreover, we say that the permutation group {\em $G$ maps $\cA$ onto $\cB$} if, for all $\gr\in G$ not the identity, $\gr$ maps $\cA$ onto $\cB$. In particular, a permutation group is imprimitive if there exists a non-trivial partition which is mapped onto itself by $G$.
\end{definitio}

\begin{definitio}
 A partition $\cA$ of $V$ is said {\em linear} if there exists a subspace $U$  of $V$ such that $$\cA=\{ U+v\mid v\in V\}.$$ We denote with $\cL(U)$ such a partition.
\end{definitio}

For the translation group, we have the following characterization of the possible partitions $\cA$ and $\cB$ such that $T(V)$ maps $\cA$ onto $\cB$.

\begin{propositio}[\cite{harpes1997partitioning}]\label{prop:partTV}
Let $\cA$ and $\cB$ be two partitions of $V=(\FF_2)^n$. Then the permutation group $T(V)$ maps $\cA$ onto $\cB$ if and only if $\cA=\cB$ and $\cA$ is a linear partition.
\end{propositio}

Focusing on the mixing-layer we have.

\begin{propositio}[\cite{bannier2016partition}]\label{prop:ml}
Let $\gl$ be a linear permutation of $V$ and let $U$ be a subspace of $V$. Then $\cL(U)\gl=\cL(U\gl)$.
\end{propositio}

For tb ciphers with independent round keys we have the following.

\begin{theore}[\cite{bannier2016partition,calderini2018}]\label{th:francesi}
Let $\cC$ be a translation based cipher on $V$. Suppose that there exist $\cA$ and $\cB$ non-trivial partitions such that for all $\ell$-tuples of round-keys $k=(k_1,...,k_\ell)$ the encryption function $\gt_k$ maps $\cA$ onto $\cB$. Define $\cA_1=\cA$ and for $1\le i\le \ell$, $\cA_{i+1}=\cA_{i}\gt_{i}$, where $\gt_i=\gamma_i\gl_i$ is the $i$-th round function without the round key translation, and suppose also that $\gt_1$ is the identity map. Then
\begin{itemize}
\item $\cA_{\ell+1}=\cB$
\item for any $1\le i\le \ell+1$, $\cA_{i}$ is a linear partition.
\end{itemize}
\end{theore}
\begin{remar}
The proof of Theorem~\ref{th:francesi} is based on the observation that since we have all possible round keys for all rounds, fixed any round $i$, all the translations $\gs_{k_i}$ (the XOR with the $i$th round key) map $\cA_i$ onto the same partition. So, from Proposition \ref{prop:partTV} we have that $\cA_i$ is linear.
\end{remar}
\begin{remar}
When $\gt_1$ is the identity, the first round is represented by a whitening (i.e. adding only the round key) and the only partitions which can be mapped by $T(V)$ into another partition are the linear partitions (Proposition \ref{prop:partTV}). However, if $\gt_1$ is not the identity map, then we might have that $\cA_1$ is non-linear, indeed it may happen that the first parallel S-box maps a non-linear partition $\cA_1$ into the partition $\cA'=\cA_1\gamma_1$. But then, $\cA_2=\cA'\gl_1$ has to be linear (and so $\cA'$) since for the first round key we can have all possible translations.

Moreover, from the proof of Theorem~\ref{th:francesi} (see  \cite{bannier2016partition}) we can note that the maps $\gt_i$'s can be any permutation in $\Sym(V)$ and not only compositions of a parallel S-boxes and a linear permutations.
\end{remar}

From \cite{calderini2018}, we have the following sufficient conditions to avoid the partition trapdoor in $\gG_{ind}(\cC)$.
%
%
%

%

\begin{theore}\label{th:main}
Let $\cC$ be a tb cipher with $\ell$ rounds. Let $h<\ell$ be such that the mixing layer $\gl_h$ is strongly proper
and the parallel S-boxes of round $h$ and round $h+1$, $\gamma_h$ and $\gamma_{h+1}$, are composed by S-boxes which are
\begin{enumerate}
\item  differentially $2^r$-uniform, with $r<m$, 
\item strongly $(r-1)$-anti-invariant.
\end{enumerate}
 Then, there do not exist $\cA$ and $\cB$ non-trivial partitions such that for all $\ell$-tuple of round keys the encryption functions map $\cA$ onto $\cB$. In particular, $\gG_{ind}(\cC)$ is primitive.
\end{theore}

\section{On algebraic properties of $\Gamma(\cC)$}\label{sec:gamma}

In this section, we focus on the group $\Gamma(\cC)$. Studying the group $\Gamma(\cC)$ is not an easy task, because such a group depends strongly on the key-schedule algorithm. 
For this reason, most authors have preferred to study the group $\Gamma_\infty(\cC)$, and, more recently, the group $\Gamma_{ind}(\cC)$.

We will give two properties for the algorithm of a key-schedule, which permit to provide sufficient conditions for avoiding partition trapdoors.
The first requirement for the key-schedule permits to avoid the trapdoor using the same conditions as in Theorem~\ref{th:main}.
In the following we will denote by $\Phi$ the key-schedule algorithm, which is the map from $\cK$ to $V^\ell$, associating a master key $k$ to a round keys $(k_1,...,k_\l)$.

\begin{definitio}\label{def:key}
Let $\cC$ be a tb cipher over $V$, with key space $\cK$. Let $\ell$ be the number of the rounds and $\Phi:\cK\to V^\ell$ be the key-schedule of $\cC$. Suppose that there exists $2\le i\le \ell-1 $ such that for any $(k_{i-1},k_i,k_{i+1})\in V^3$ the element 
$(\bar{k}_1,...,\bar{k}_{i-2},k_{i-1},k_i,k_{i+1},\bar{k}_{i+2},...,\bar{k}_\ell)$ is in $\Im(\Phi)$,  where the values $\bar{k}_j$'s are fixed. Then we say that the key-schedule $\Phi$ is {\em 3-rounds independent at round $i$}.
\end{definitio}

We need the following result from \cite{calderini2018}.

\begin{propositio}\label{prop:blocchi}
Let $\gamma$ be a parallel map, i.e. $\gamma=(\gamma_1,...,\gamma_b)$ with $\gamma_i\in\Sym(V_i)$ for all $i$. Suppose that for all $i$ the map $\gamma_i$ is
\begin{enumerate}
\item  differentially $2^r$-uniform, with $r<m$, 
\item strongly $(r-1)$-anti-invariant. 
\end{enumerate}
Let $\cL(U)$ and $\cL(W)$ be non-trivial linear partitions of $V$, then $\gamma$ maps $\cL(U)$ onto $\cL(W)$ if and only if $U$ and $W$ are walls, in particular $U=W$.
\end{propositio}

\begin{theore}\label{th:gam}
Let $\cC$ be a tb cipher with a key-schedule $\Phi$ 3-round independent at round $i$, for some $2\le i\le \ell-1 $. Then, if the parallel S-boxes of round $i$ and round $i+1$, $\gamma_i$ and $\gamma_{i+1}$, satisfy Condition 1) and 2) of Proposition \ref{prop:blocchi} and $\gl_i$ is strongly proper, there do not exist $\cA$ and $\cB$ non-trivial partitions such that for all $k\in\cK$ the map $\tau_k$ maps $\cA$ onto $\cB$. In particular, $\gG(\cC)$ is primitive and it is not weak with respect to the partition-based trapdoor.
\end{theore}
\begin{proof}
Since $\Phi$ is 3-round independent at round $i$, there exist some fixed values $\bar{k}_j$, with $j\in \{1,...,\ell\}\setminus\{i-1,i,i+1\}$, such that 
$(\bar{k}_1,...,k_{i-1},k_i,k_{i+1},...,\bar{k}_\ell)$ is in $\Im(\Phi)$, for any $(k_{i-1},k_i,k_{i+1})\in V^3$.\\
Let us consider the encryption maps $$\tau_k=\tau_{1}\gs_{\bar{k}_1}\cdots\tau_{i-2}\gs_{\bar{k}_{i-2}}\tau_{i-1}\gs_{{k}_{i-1}}\tau_{i}\gs_{{k}_{i}}\tau_{i+1}\gs_{{k}_{i+1}}\tau_{i+2}\gs_{\bar{k}_{i+2}}\cdots\tau_{\ell}\gs_{\bar{k}_{\ell}},$$
where $\gt_j=\gamma_j\gl_j$ for all $j$.

Denoting by $\bar{\gt}_1=\tau_{1}\gs_{\bar{k}_1}\cdots\tau_{i-2}\gs_{\bar{k}_{i-2}}\tau_{i-1}$ and by $\bar{\gt}_2=\tau_{i+2}\gs_{\bar{k}_{i+2}}\cdots\tau_{\ell}\gs_{\bar{k}_{\ell}}$, we have that these maps are given by
\begin{equation}\label{eq:gtK}
\bar{\gt}_1 \gs_{k_{i-1}}\gamma_i\gl_i\gs_{k_i}\gamma_{i+1}\gl_{i+1}\gs_{k_{i+1}}\bar{\gt}_2,
\end{equation}
where $(k_{i-1},k_i,k_{i+1})\in V^3$.
Suppose that there exist $\cA$ and $\cB$ non-trivial partitions such that any encryption function $\tau_k$ (and in particular those as in \eqref{eq:gtK}) maps $\cA$ onto $\cB$.
We can consider the maps in \eqref{eq:gtK} as the encryption functions of a cipher having 3 rounds with independent round keys.

Denote by $\cA_1=\cA\bar\gt_1$, $\cA_2=\cA_1\gamma_i\gl_i$, $\cA_3=\cA_2\gamma_{i+1}\gl_{i+1}$. From Theorem~\ref{th:francesi} we have that $\cA_1$, $\cA_2$ and $\cA_3$ are linear. Thus, there exist three vector subspace $W_1,W_2$ and $W_3$ (non-trivial) such that $\cL(W_j)=\cA_j$ for $j=1,2,3$.\\
 From Proposition \ref{prop:blocchi} it follows that $W_1, W_2$ and $W_3$ are walls. But, this implies that $W_2=W_1\gl_i$ (see Proposition \ref{prop:ml}) which is not possible since $\gl_i$ is strongly-proper. So, there do not exist $\cA$ and $\cB$ non-trivial.
\end{proof}

In the following, we want to consider key-schedules for which the property in Definition~\ref{def:key} does not hold for all possible round keys in three consecutive rounds.
In particular, we will consider the case when for three (consecutive) rounds the keys are contained in some subgroup $T'\subseteq T((\FF_2)^n)$ (maybe not the same for all the three rounds) of order $2^{n-1}$ (half of the possible round keys).

In order to study this case, we need to characterize the pairs of partitions $\cA,\cB$ for which $T'$ maps $\cA$ onto $\cB$. Before giving such a characterization, we will define the following partition of $(\FF_2)^n$.

\begin{definitio}
Let $U$ be a vector subspace of $(\FF_2)^n$ of dimension $n-1$. Let $W_1$ and $W_2$ be two distinct subspaces of $U$. We define the partition of $(\FF_2)^n$
$$
\cL\cA_U(W_1|W_2)=\{W_1+v\,:\,v \in U\}\cup\{(W_2+\bar v)+v\,:\,v \in U\}
$$
for any arbitrary element $\bar v \in (\FF_2)^n\setminus U$. We will call $\cL\cA_U(W_1|W_2)$ a {\em linear-affine partition}.
\end{definitio}
Note that if $W_1=W_2$, then the partition is linear.
\begin{remar}
For every $\bar v \in (\FF_2)^n\setminus U$, the space  $(\FF_2)^n$ can be partitioned into $U$ and $U+\bar v$. Then a linear-affine partition $\cL\cA_U(W_1|W_2)$ is just a linear partition restricted to $U$, $\cL(W_1)$, plus another linear partition $\cL(W_2)$, always restricted to $U$, translated by the vector $\bar v$. That is, $\cL(W_1)$ is a partition for $U$ and $\cL(W_2)+\bar v$ is a partition for $U+\bar v$.
\end{remar}

\begin{propositio}\label{prop:blocchi2}
Let $V=(\FF_2)^n$ and $T'$ be a subgroup of $T(V)$ of order $2^{n-1}$. Let $\cA$ and $\cB$ be two partitions of $V$. Then, the permutation subgroup $T'$ maps $\cA$ onto $\cB$ if and only if 
\begin{itemize}
\item $\cA=\cB$ and 
\item $\cA$ is a linear partition, $\cL(W)$, or $\cA$ is linear-affine partition $\cL\cA_U(W_1|W_2)$, where $U=\{v\,:\,\gs_v\in T'\}$.
\end{itemize}
\end{propositio}
\begin{proof}
It is easy to check that if $\cA$ is linear or linear-affine, then it is mapped by $T'$ into itself.

Viceversa, let us consider two partitions $\cA$ and $\cB$ and suppose that $T'$ maps $\cA$ onto $\cB$. Let $U=\{v\,:\,\gs_v\in T'\}$. Since $T'$ is a subgroup of order $2^{n-1}$, $U$ has dimension $n-1$. \\ Denote by $A_0\in\cA$ and $B_0\in\cB$ the blocks of the partitions which contain $0$.
We can have two cases 
\begin{enumerate}
\item $A_0\subseteq U$
\item $\exists \bar v\in A_0$ such that $\bar v\notin U$
\end{enumerate}

Consider the first case. The block $A_0$ is a vector subspace. Indeed, for any $v\in A_0$ we have $0\in A_0\gs_v$ and thus $ A_0\gs_v=B_0$. Moreover, we have also $v \in B_0$. This implies $A_0=B_0$ and $ A_0\gs_v=A_0$ for all $v$ in $A_0$. So, $A_0$ is a vector subspace.\\
 Then, for all $v$ in $U$ the set $A_0+v$ is a block of $\cB$. These blocks form a partition for $U$ and, obviously, $A_0+v$ is also a block of $\cA$ since $\gs_v^{-1}=\gs_v$.
 
Now, in a similar way, we can fix $\bar v\notin U$ and consider the blocks $A_{\bar{v}}\in\cA$ and $B_{\bar{v}}\in\cB$ containing the vector $\bar v$. Since $\{A_0+v\}_{v\in U}$ is a partition for $U$ we have that  $A_{\bar{v}}\subseteq U+\bar v$ and also $B_{\bar{v}}\subseteq U+\bar v$.\\
Let $A'=A_{\bar{v}}+\bar v$ and $B'=B_{\bar{v}}+\bar v$. The sets $A',B'\subseteq U$ are vector subspaces and $A'=B'$. Indeed, for any $v \in A'$ we have $A_{\bar{v}}\gs_v= B_{\bar{v}}$, since $\bar v\in A_{\bar{v}}\gs_v$. Moreover,  $v+\bar v \in B_{\bar{v}}$ and, since $B_{\bar{v}}\gs_v= A_{\bar{v}}$, it follows that $v+\bar v \in A_{\bar{v}}$. Thus, $A_{\bar{v}}=B_{\bar{v}}$ and it is an affine subspace, implying that $A'=B'$ and $A'$ is a vector subspace.

As above, for all $v$ in $U$ the set $A_{\bar{v}}+v$ is a block of $\cA$ and $\cB$. So, $\{A_{\bar{v}}+v\}_{v\in U}$ forms a partition of $U+\bar v$ and $\cA=\cB$. Now, if $A'=A_0$ then $\cA$ is linear, otherwise $\cA=\cL\cA_U(A_0|A')$.

For the second case, we have that there exists $\bar v\in A_0$ such that $\bar v\notin U$. Let $W=A_0\cap U$. 
For any $v\in W$, since $0\in A_0\gs_v$ we have that $A_0\gs_v=B_0$ and also that $v\in B_0$. Moreover, we have that $B_0\gs_v=A_0$ for all $v$ in $W$. Which implies that for any $v_1,v_2 \in W$, $v_1+v_2\in A_0$. Since $v_1+v_2$ is also in $U$, $v_1+v_2\in W$ and thus $W$ is a vector subspace. \\
In a similar way, we have also that $W+\bar v\subseteq A_0$ and $W+\bar v\subseteq B_0$. 
Now, suppose that $W \cup (W+\bar v)\subsetneq A_0$. Then, there exists $w\in A_0\cap (U+\bar v)$ with $w\notin W+\bar v$. Thus, $w=\bar v+u$ for some $u$ in $U\setminus W$. But then, $A_0\gs_u=B_0$ since $\bar v\in B_0$. Moreover, $B_0\gs_u=A_0$ implies that $u\in A_0$ which is not possible. Thus, $A_0=W \cup (W+\bar v)= B_0$. So, $A_0$ is a vector subspace and it is easy to check that $\{A_0+v\,:\,v \in U\}$ is a partition of $(\FF_2)^n$ and $\cA=\cB=\cL(A_0)$.
\end{proof}

For a given permutation $\g$ which maps a linear (linear-affine) partition into another one we have the following properties.

\begin{lemm}\label{lemmaeasy}
Let $\gamma\in \Sym(V)$ with $0\gamma=0$. Then:
\begin{itemize}
\item[(i)] If $\gamma$ maps $\cL(W)$ onto $\cL(W')$, then for any $u\in W$ we have $\hat{\gamma}_u(V)\subseteq W'$.
\item[(ii)] If $\gamma$ maps $\cL(W)$ onto $\cL\cA_{U}(W_1|W_2)$, then
\begin{itemize}
\item for any $u\in W$, we have $\hat{\gamma}_u(S)\subseteq W_1$, where $S=U\gamma^{-1}$;
\item for any $u\in W$, we have $\hat{\gamma}_u(S')\subseteq W_2$, where $S'=(U+\bar v)\gamma^{-1}$ and $\bar v$ is an arbitrary element of $V\setminus U$.
\end{itemize}
\item[(iii)] If $\gamma$ maps $\cL\cA_{U}(W_1|W_2)$ onto $\cL(W)$, then 
\begin{itemize}
\item for any $u\in W_1$, we have $\hat{\gamma}_u(U)\subseteq W$; 
\item for any $u\in W_2$, we have $\hat{\gamma}_u(U+\bar v)\subseteq W$, where $\bar v$ is an arbitrary element of $V\setminus U$.
\end{itemize}
\item[(iv)] If $\gamma$ maps $\cL\cA_U(W_1|W_2)$ onto $\cL\cA_{U'}(W'_1|W'_2)$, then
\begin{itemize}
\item for any $u\in W_1$, we have $\hat{\gamma}_u(S)\subseteq W'_1$, where $S=U\cap U'\gamma^{-1}$; 
\item for any $u\in W_2$, we have $\hat{\gamma}_u(S')\subseteq W'_1$, where $S'=(U+\bar v)\cap U'\gamma^{-1}$ and $\bar v$ is an arbitrary element of $V\setminus U$.
\end{itemize}
On the other hand,
\begin{itemize}
\item  for any $u\in W_1$, we have $\hat{\gamma}_u(S)\subseteq W'_2$, where $S=U\cap (U'+\bar v ')\gamma^{-1}$, and $\bar v'$ is an arbitrary element of $V\setminus U'$;
\item  for any $u\in W_2$, we have $\hat{\gamma}_u(S')\subseteq W'_2$, where $S'=(U+\bar v)\cap (U'+\bar v ')\gamma^{-1}$ and  $\bar v$, $\bar v'$ are arbitrary elements of $V\setminus U$ and $V\setminus U'$, respectively.
\end{itemize}
\end{itemize}
\end{lemm}
\begin{proof}~\\
\noindent{\bf (i)}
Since $\gamma$ maps $\cL(W)$ onto $\cL(W')$ and $0\gamma=0$ we have $W\gamma=W'$. Moreover, for any $v \in V$ we deduce that $(W+v)\gamma=W'+v\gamma$. Then, let $u\in W$, for any $v\in V$ $(v+u)\gamma\in W'+v\gamma$ and $(v+u)\gamma+v\gamma=\hat\gamma_u(v)\in W'$.\\

\noindent{\bf (ii)}
Let $S=U\gamma^{-1}$. For any $s \in S$, the element $s\gamma$ is contained in  $W_1+s\gamma\subseteq U$. Then, for any $s \in S$ we have $(W+s)\gamma =W_1+s\gamma$ and, thus, $(u+s)\gamma+s\gamma=\hat\gamma_u(s)\in W_1$. On the other hand, consider $S'=(U+\bar v)\gamma^{-1}$. For any $s \in S'$ it holds $s\gamma=w+\bar v$ for some $w\in U$ and $(W+s)\gamma=W_2+s\gamma=(W_2+\bar v)+w$. Thus, let $u\in W$, for  any $s \in S'$ we obtain $\hat\gamma_u(s)\in W_2$. \\

\noindent{\bf (iii)}
Now we have $W_1\gamma=W$ and $(W_2+\bar v)\gamma =W+\bar v$. Let $u\in W_1$, since for any $v\in U$ it holds $(W_1+v)\gamma=W+v\gamma$, we have $\hat\gamma_u(v)\in W$. On the other hand, if we consider $v\in U+\bar v$, we have $v=w+\bar v$ for some $w\in U$, and $((W_2+\bar v)+w)\gamma=W+v\gamma$. This implies that, fixed $u\in W_2$, for any $v\in U+\bar v$ we have $\hat\gamma_u(v)\in W$.\\

\noindent{\bf (iv)}
Let $S=U\cap U'\gamma^{-1}$. For any $s\in S$ we have $s\in W_1+s\subseteq U$ and $(W_1+s)\gamma=W'_1+s\gamma\subseteq U'$. As before, for any fixed $u\in W_1$, we obtain $\hat\gamma_u(s)\in W'_1$ for any $s\in S$.
Similarly, let $S'=(U+\bar v)\cap U'\gamma^{-1}$. For any $s\in S'$ we have $s=(w+\bar v)$, with $w\in U$, and  $s\in (W_2+\bar v)+w$. So, $((W_2+\bar v)+w)\gamma=W'_1+s\gamma$. Let $u\in W_2$ be fixed, we obtain $\hat\gamma_u(s)\in W'_1$ for any $s\in S'$.

On the other hand, consider $S=U\cap (U'+\bar v ')\gamma^{-1}$. Now, for any $s\in S$ we have $s\in W_1+s$ and $s\gamma=w'+\bar v'$ for some $w'\in U'$. Then, $(W_1+s)\gamma=(W'_2+\bar v')+w'=W'_2+s\gamma$. Fixed $u\in W_1$, we obtain $\hat\gamma_u(s)\in W'_2$ for all $s\in S$.\\ 
In a similar way, we have that, letting $S'=(U+\bar v)\cap (U'+\bar v ')\gamma^{-1}$ and $u\in W_2$, $\hat{\gamma}_u(s)\subseteq W'_2$ for any element $s$ in $S'$.
\end{proof}

Let $V=V_1\oplus ...\oplus V_b$, $\dim(V)=n$ and $V_i=(\FF_2)^m$. For any subspace $U$ of dimension $n-1$ we can define the set $$J_U=\{j\,:\, V_j\cap U\subsetneq V_j\}.$$ Note that, since $\dim(U)=n-1$ for any $j\in J_U$ we have  $\dim(U\cap V_j)=m-1$.\\

In the following, we will characterize the possibles linear and linear-affine partitions which can be mapped one into another by a parallel map, satisfying some cryptographic conditions. In particular, given a parallel map $\gamma$, we need to consider the cases:
\begin{enumerate}
\item  $\gamma$ maps a linear partition onto a linear-affine partition.
\item $\gamma$ maps a linear-affine partition onto a linear partition.
\item $\gamma$ maps a linear-affine partition onto a linear-affine partition.
\end{enumerate}
Note that the case when $\gamma$  maps a linear partition onto a linear partition has been dealt with in Proposition \ref{prop:blocchi}.

\begin{lemm}\label{lm:1}
Let $\gamma=(\gamma_1,...,\gamma_b)$ be a parallel map over $V=V_1\oplus ...\oplus V_b$, with $V_i=(\FF_2)^m$, $0\gamma=0$. Suppose that for all $i$ the map $\gamma_i$ is
\begin{enumerate}
\item  differentially $2^r$-uniform, with $r<m-1$, 
\item strongly $r$-anti-invariant. 
\end{enumerate}
If $\gamma$ maps $\cL(W)$ onto $\cL\cA_{U}(W_1|W_2)$ (non-trivial), then $W$, $W_1$ and $W_2$ are walls and $W=W_1=W_2$. In particular, $\cL\cA_{U}(W_1|W_2)$ is linear.
\end{lemm}
\begin{proof}
Let $I=\{i\,:\, \pi_i(W)\ne \{0\}\}$. Clearly $I\ne\emptyset$ since $W$ is not trivial. Suppose that there exists $i\in I$ such that $V_i\cap W\subsetneq V_i$, i.e. $W$ is not a wall.

First of all, there exists $u\in V_i\cap W$ nonzero. Indeed, consider any $u\in W$ with $\pi_i(u)\ne 0$ (such an element exists since $i\in I$). 
Let $S=U\gamma^{-1}$, since $\gamma$ is a parallel map and $0\gamma=0$ we have that $S'=S\cap V_i=(U\cap V_i)\gamma^{-1}$. Moreover, $|S'|\ge 2^{m-1}$ ($|U\cap V_i|=2^{m-1}$ if $i \in J_U$). From Lemma~\ref{lemmaeasy} it follows that $\hat \gamma_u(S')\subseteq W_1$ and $|\hat \gamma_u(S')|\ge |S'|/2^r\ge 2$. 

Recalling that $0\gamma=0$, we have $W\gamma=W_1$ and  $u\gamma\in W_1$. Thus, since $W_1$ is a vector subspace it follows that $\hat \gamma_u(S')+u\gamma\subseteq W_1$. Actually, $\hat \gamma_u(S')+u\gamma\subseteq W_1\cap V_i$. So, there exists $w\ne 0$ in $W_1\cap V_i=(W\cap V_i)\gamma$, implying that there exists $u$ nonzero in $W\cap V_i$. 

Now, let $0\ne u\in W\cap V_i$. We obtain that $\hat\gamma_u(S')=\hat\gamma_{i,u}(S')\subseteq W_1\cap V_i$, which implies $|W_1\cap V_i| \ge 2^{m-r-1}+1$ (note that $|\hat\gamma_{i,u}(S')|\ge 2^{m-r-1}$ and $0\notin\hat\gamma_{i,u}(S')$ since $\gamma_i$ is a permutation). Thus, $\dim(W_1\cap V_i)\ge m-r$ which is in contradiction with the fact that $\gamma_i$ is strongly $r$-anti-invariant.

So, $W$ is a wall and since $\gamma$ is a parallel map it holds $W=W_1$. Moreover, $W\subseteq U$ implies that for all $i\in J_U$ we have that $V_i\nsubseteq W$.\\
 Fix any $i\in J_U$, we can choose $\bar v\in V_i$ and $\bar v \notin U$. Denoting by $v=\bar v \gamma^{-1}$, we have $(W+v)\gamma=W_2+\bar v$. Since $W$ is a wall, $\gamma$ is a parallel map and $v,\bar v \in V_i$ it is easy to check that also $W_2$ is a wall and $W_2=W$. Since $W_1=W_2$ the partition $\cL\cA_{U}(W_1|W_2)$ is linear.
\end{proof}

If we know a priori the $n-1$ dimensional vector space $U$, linked to the linear-affine partition, then we can obtain the following corollary.
\begin{corollary}
Let $\gamma=(\gamma_1,...,\gamma_b)$ be a parallel map over $V=V_1\oplus ...\oplus V_b$, with $V_i=(\FF_2)^m$, $0\gamma=0$. Let $U$ be a vector space of dimension $n-1$ and suppose that for all $i\notin J_U$ the map $\gamma_i$ is
\begin{enumerate}
\item  differentially $2^r$-uniform, with $r<m$, 
\item strongly $r-1$-anti-invariant. 
\end{enumerate}
otherwise for $j \in J_U$ $\gamma_j$  is
\begin{enumerate}
\item[(1')]  differentially $2^r$-uniform, with $r<m-1$, 
\item[(2')] strongly $r$-anti-invariant. 
\end{enumerate}
If $\gamma$ maps $\cL(W)$ onto $\cL\cA_{U}(W_1|W_2)$ (non-trivial), then $W$, $W_1$ and $W_2$ are walls and $W=W_1=W_2$. In particular, $\cL\cA_{U}(W_1|W_2)$ is linear.
\end{corollary}
\begin{proof}
Let $I=\{i\,:\, \pi_i(W)\ne \{0\}\}\ne\emptyset$. Suppose that there exists $i\in I$ such that $V_i\cap W\subsetneq V_i$.
In this case we can distinguish if $i\notin J_U$ or not. If $i\notin J_U$, then the set $S'=(U\cap V_i)\gamma^{-1}=V_i$. Following the same steps as above, we have $\hat \gamma_{i,u}(S')\subseteq W_1\cap V_i$, and so $|W_1\cap V_i| \ge 2^{m-r}+1$. Then, the $r-1$ strong anti-invariance  is sufficient to obtain a contradiction.

On the other hand, if $i\in J_U$ the proof follows from Lemma~\ref{lm:1}.
\end{proof}

\begin{lemm}\label{lm:2}
Let $\gamma=(\gamma_1,...,\gamma_b)$ be a parallel map over $V=V_1\oplus ...\oplus V_b$, with $V_i=(\FF_2)^m$, $0\gamma=0$. Suppose that for all $i$ the map $\gamma_i$ is
\begin{enumerate}
\item  differentially $2^r$-uniform, with $r<m-1$, 
\item strongly $r$-anti-invariant. 
\end{enumerate}
If $\gamma$ maps $\cL\cA_{U}(W_1|W_2)$ onto $\cL(W)$ (non-trivial), then $W_1$, $W_2$ and $W$ are walls and $W_1=W_2=W$. In particular, $\cL\cA_{U}(W_1|W_2)$ is linear.
\end{lemm}
\begin{proof}
Let $I=\{i\,:\, \pi_i(W_1)\ne \{0\}\}$. Clearly $I\ne\emptyset$ since $W_1\gamma=W$ and the space $W$ is not trivial. Suppose that there exists $i\in I$ such that $V_i\cap W_1\subsetneq V_i$.
First of all, there exists $u\in V_i\cap W_1$ nonzero. Indeed, consider any $u\in W_1$ with $\pi_i(u)\ne 0$.
Let $S'=U\cap V_i$, which is a subspace of dimension $m-1$ and thus $|S'|=2^{m-1}$. From Lemma~\ref{lemmaeasy} we have that $\hat \gamma_u(S')\subseteq W$ and $|\hat \gamma_u(S')|\ge |S'|/2^r\ge 2$. \\
Since $u\gamma\in W$ we have that $\hat \gamma_u(S')+u\gamma\subseteq W\cap V_i$ and thus there exists $w\ne 0$ in $W\cap V_i=(W_1\cap V_i)\gamma$, implying that there exists $u$ nonzero in $W_1\cap V_i$. 

Then, let $u\ne 0$ in $W_1\cap V_i$, we have $\hat\gamma_u(S')=\hat\gamma_{i,u}(S')\subseteq W\cap V_i$, which implies $|W\cap V_i| \ge 2^{m-r-1}+1$. So, $\dim(W\cap V_i)\ge m-r$, which is not possible since $\gamma_i$ is strongly $r$-anti-invariant. Thus, $W_1$ is a wall and the same for $W=W_1\g$.

Similarly to Lemma~\ref{lm:1}, we can select $\bar v \in V_i$ for some $i\in J_U$. Also in this case $V_i\nsubseteq W$. Then, $(W_2+\bar v)\gamma=W+\bar v \gamma$. From the fact that $W$ is a wall and $\gamma$ is a parallel map we have that $W_2$ is a wall and $W_2=W$. From this follows that $\cL\cA_{U}(W_1|W_2)$ is a linear partition.
\end{proof}

As before, using the knowledge of the $n-1$ dimensional vector space $U$, we can obtain the following corollary.
\begin{corollary}
Let $\gamma=(\gamma_1,...,\gamma_b)$ be a parallel map over $V=V_1\oplus ...\oplus V_b$, with $V_i=(\FF_2)^m$, $0\gamma=0$. Let $U$ be a vector space of dimension $n-1$ and suppose that for all $i\notin J_U$ the map $\gamma_i$ is
\begin{enumerate}
\item  differentially $2^r$-uniform, with $r<m$, 
\item strongly $r-1$-anti-invariant. 
\end{enumerate}
otherwise for $j \in J_U$ $\gamma_j$
\begin{enumerate}
\item[(1')]  differentially $2^r$-uniform, with $r<m-1$, 
\item[(2')] strongly $r$-anti-invariant. 
\end{enumerate}
If $\gamma$ maps $\cL\cA_{U}(W_1|W_2)$ onto $\cL(W)$ (non-trivial), then $W_1$, $W_1$ and $W$ are walls and $W_1=W_2=W$. In particular, $\cL\cA_{U}(W_1|W_2)$ is linear.
\end{corollary}

Now, we have to analyze the last case, that is, when a linear-affine partition is mapped into another linear-affine partition by a parallel map.

\begin{lemm}\label{lm:3}
Let $\gamma=(\gamma_1,...,\gamma_b)$ be a parallel map over $V=V_1\oplus ...\oplus V_b$, with $V_i=(\FF_2)^m$, $0\gamma=0$.. Suppose that for all $i$ the map $\gamma_i$ is
\begin{enumerate}
\item  differentially $2^r$-uniform, with $r<m-1$, 
\item strongly $r$-anti-invariant. 
\end{enumerate}

If $\gamma$ maps $\cL\cA_{U}(W_1|W_2)$ onto  $\cL\cA_{U'}(W'_1|W'_2)$ (non-trivial), with $U$ and $U'$ such that $J_U\cap J_{U'}= \emptyset$. Then, all the spaces $W_1, W'_1, W_2, W'_2$ are walls and $W_1=W'_1=W_2=W_2'$. In particular, both partitions are linear.
\end{lemm}
\begin{proof}
Let $I=\{i\,:\, \pi_i(W_1)\ne \{0\}\}$. Also in this case $I\ne\emptyset$. Indeed, $W_1\gamma=W_1'$ and $J_U\cap J_{U'}=\emptyset$ implies that there exists $i\in J_U$ and $v\in U\cap V_i$ such that $(W_1+v)\gamma=W_2'+v\gamma$. So, $|W_1|=|W_1'|=|W_2'|$ which implies that $W_1$ cannot be trivial.

 Suppose that there exists $i\in I$ such that $V_i\cap W_1\subsetneq V_i$.
First of all, there exists $u\in V_i\cap W_1$ nonzero. Indeed, consider any $u\in W_1$ with $\pi_i(u)\ne 0$. 
Suppose $U\cap V_i\ne V_i$ (i.e. $i$ is in $J_U$ but not in $J_{U'}$). Then, $(U'\cap V_i)=V_i$ and $S=(U'\cap V_i)\gamma^{-1}=V_i$, implying $S'=S\cap U\cap V_i=U\cap V_i$ (the case $U\cap V_i= V_i$ and $U'\cap V_i\ne V_i$ is similar). From Lemma~\ref{lemmaeasy} we have that $\hat \gamma_u(S')\subseteq W'_1$ and $|\hat \gamma_u(S')|\ge |S'|/2^r\ge 2$. Since $u\gamma\in W'_1$ (recall that $W_1\gamma=W'_1$) we have that $\hat \gamma_u(S')+u\gamma\subseteq W'_1\cap V_i$ and thus there exists $w\ne 0$ in $W'_1\cap V_i=(W_1\cap V_i)\gamma$, implying that there exists $u$ nonzero in $W_1\cap V_i$. \\
Then, let $u\ne 0 \in W_1\cap V_i$,  we have $\hat\gamma_u(S')=\hat\gamma_{i,u}(S')\subseteq W'_1\cap V_i$, which implies $|W'_1\cap V_i| \ge 2^{m-r-1}+1$ and so $\dim(W'_1\cap V_i)\ge m-r$. This is not possible since $\gamma_i$ is strongly $r$-anti-invariant.

Also in this case, since $\gamma$ is a parallel map we have $W_1=W_1'$. Note that if $i\in J_U\cup J_{U'}$ then $V_i\nsubseteq W_1$. Moreover, selecting $\bar v\in V_i$ not in $U$, for some $i\in J_U$, since $V_i\subset U'$ we have $(W_2+\bar v)\gamma=W_1+\bar v\gamma$. As in Lemma~\ref{lm:2} we obtain $W_2=W_1'$. On the other hand, selecting $\bar v'\in V_i$ for some $i\in J_{U'}$ we can follow the steps in the proof of Lemma~\ref{lm:1} and obtain $W_2'=W_1$. From this we have that the two partitions are linear.
\end{proof}

If we know a priori the $n-1$ dimensional vector spaces $U$ and $U'$ then we can obtain the following.
\begin{corollary}
Let $\gamma=(\gamma_1,...,\gamma_b)$ be a parallel map over $V=V_1\oplus ...\oplus V_b$, with $V_i=(\FF_2)^m$, $0\gamma=0$.. Let $U, U'$ be two vector spaces of dimension $n-1$ and let $J_U$ and $J_{U'}$ as above, with $J_U\cap J_{U'}= \emptyset$. Suppose that for all $i\notin J_U\cup J_{U'}$ the map  $\gamma_i$ is
\begin{enumerate}
\item  differentially $2^r$-uniform, with $r<m$, 
\item strongly $r-1$-anti-invariant. 
\end{enumerate}
and for $j\in J_U\cup J_{U'}$ $\gamma_{j}$  is
\begin{enumerate}
\item[(1')] differentially $2^r$-uniform, with $r<m-1$, 
\item[(2')] strongly $r$-anti-invariant. 
\end{enumerate}
If $\gamma$ maps $\cL\cA_{U}(W_1|W_2)$ onto  $\cL\cA_{U'}(W'_1|W'_2)$ (non-trivial), with $U$ and $U'$ such that $J_U\cap J_{U'}= \emptyset$. Then, all the spaces $W_1, W'_1, W_2, W'_2$ are walls and $W_1=W'_1=W_2=W_2'$. In particular, both partitions are linear.
\end{corollary}

From the cryptographic properties used in the previous results, we have that if we consider the case of an APN vectorial Boolean function, that is a differentially $2$-uniform function, then the strong $1$-anti-invariance of the S-boxes $\g_i$'s is sufficient to obtain the characterization of the partitions above.  As shown in \cite{calderinisalerno}, strong $1$-anti-invariance is equivalent to having nonlinearity different from $0$, which is always true for APN functions (see \cite[Proposition~9.15]{carlet}).
However, for even values of $m$, it is known that for $m=4$ there exist no APN permutations \cite{calapn,hou}, for $m=6$ there exists only one APN permutation \cite{dillon} (up to equivalence), and for $m\ge 8$ no APN permutations are known.
So, for $m$ even (except $m=6$), differentially $4$-uniform (bijective) S-boxes are optimal with respect to the differential uniformity.

For applying the previous results, in the case of differentially $4$-uniform functions we need that the S-boxes are also strongly $2$-anti-invariant. However, for some dimensions, such as $m=4$, we have no differentially $4$-uniform function which is strongly $2$-anti-invariant. \\
Let us note that for $m=4$ there exist permutations which are strongly $2$-anti-invariant and also differentially $6$-uniform. With these functions, we would produce the same results given in the previous lemmas. Indeed, for example in the proof of Lemma~\ref{lm:3}, supposing that $W'_1$ is not a wall and considering $i$ such that $\{0\}\ne W'_1\cap V_i\subsetneq V_i$,
we would have $|W'_1\cap V_i| \ge 2^{4-1}/6+1>2$, implying $\dim(W'_1\cap V_i)\ge 2=m-2$ which is in contradiction with the strong 2-anti-invariance. However, for higher dimensions $2$-anti-invariance and differential $6$-uniformity could not be enough. Moreover, we would use known S-boxes with differential uniformity as low as possible, that in the case of even dimensions, with the exception of $m=6$, is $4$. 

For that reason in the following section we investigate the case of differentially $4$-uniform functions, individuating a different property from the anti-invariance, which is related to the linear structures of the components of an S-box. 

\subsection{The case of  differentially $4$-uniform S-boxes}\label{sec:4uni}
 
In this section, we want to focus on differentially $4$-uniform maps. We recall that for a given Boolean function $f:(\FF_2)^m\to\f2$ a nonzero vector $a\in (\FF_2)^m$ is said a {\em linear structure} of $f$ if $\hat f_a$ is constant.

Let $f:(\FF_2)^m\to(\FF_2)^m$ be vectorial Boolean function and consider the following nonlinear measure 
$$
\hat n (f):=\max_{a\in(\FF_2)^m\setminus\{0\}}|\{v\,\in(\FF_2)^m \setminus\{0\}\mid\deg(\langle\hat{f}_a,v\rangle)=0\}|.
$$
With $\hat n (f)$ we are counting for how many components of $f$ the direction $a$ is a linear structure.

Let us denote $V_a=\{v\,\in(\FF_2)^m \mid\deg(\langle\hat{f}_a,v\rangle)=0\}$. We recall the following result from \cite{weak}.

\begin{propositio}\label{prop:5}
Let $f$ be a vectorial Boolean function over $(\FF_2)^m$, and $a\in (\FF_2)^m\setminus\{0\}$. Then, $f(a)+V_a^{\perp}$ is the smallest affine subspace of $(\FF_2)^m$ containing $\Im(\hat{f}_a)$. In particular, $\hat{n}(f)=0$ if and only if there does not exist a proper affine subspace of $(\FF_2)^m$ containing $\Im(\hat{f}_a)$, for all $a\in (\FF_2)^m\setminus\{0\}$.
\end{propositio}
From Proposition~\ref{prop:5} we have that if a function $f:(\FF_2)^m\to(\FF_2)^m$ is such that $\hat n(f)=0$, then for any nonzero $a\in (\FF_2)^m$ the elements of $\Im(\hat{f}_a)$ generate the whole space $(\FF_2)^m$.

Thus we can obtain the following results.

\begin{lemm}\label{lm:11}
Let $\gamma=(\gamma_1,...,\gamma_b)$ be a parallel map over $V=V_1\oplus ...\oplus V_b$, $0\gamma=0$, with $V_i=(\FF_2)^m$ and $m\ge 4$. Suppose that for all $i$ the map $\gamma_i$ is
\begin{enumerate}
\item  differentially $4$-uniform, 
\item  $\hat{n}(\gamma_i)=0$. 
\end{enumerate}
If $\gamma$ maps $\cL(W)$ onto $\cL\cA_{U}(W_1|W_2)$ (non-trivial), then $W$ and $W_1$ are walls and $W=W_1=W_2$, implying also that $\cL\cA_{U}(W_1|W_2)$ is linear.
\end{lemm}
\begin{proof}
As before, let $I=\{i\,:\, \pi_i(W_1)\ne \{0\}\}\ne\emptyset$. Suppose that there exists $i\in I$ such that $V_i\cap W\subsetneq V_i$. If $V_i$ is also such that $V_i\cap U=V_i$ (i.e. $i\notin J_U$), then the properties on $\gamma_i$ are sufficient to obtain a contradiction. Indeed, we will have that there exist $u\ne 0$ in $W\cap V_i$ and since $V_i\cap U=V_i$, $$\hat{\gamma}_{i,u}(V_i\cap U)=\Im(\hat{\gamma}_{i,u})\subset W_1\cap V_i.$$ From Proposition \ref{prop:5}, $\hat{n}(\gamma_i)=0$ implies $W_1\cap V_i=V_i$.
Thus for all such indexes $i$ we have $V\cap V_i= W_1\cap V_i=V_i$.

Suppose that $V_i\cap U\ne V_i$, then $\dim(V_i\cap U)=m-1$.
Following the same steps of the proof of Lemma~\ref{lm:1}, we would have that there exists a nonzero $u\in W\cap V_i$ and  $\hat{\gamma}_{i,u}(S)\subset W_1\cap V_i$, with $S=(U\cap V_i)\gamma^{-1}$.\\
 Note that $S$ contains, in this case, half of the elements of $V_i$. 
Thus, $\dim(W_1\cap V_i)\ge m-2$ and also $\dim(W\cap V_i)\ge m-2$.\\
Now, since there exists $\bar v \notin U$ with $\bar v\in V_i$, we can choose $s\in V_i$ such that $s\gamma=\bar v$ and then, $(W+s)\gamma=W_2+s\gamma$. Since $\gamma$ is a parallel map we have $((W\cap V_i)+s)\gamma=(W_2\cap V_i)+s\gamma$. This implies that 
$\dim(W_2\cap V_i)\ge m-2$.
Moreover, for all $u\in W\cap V_i$ we have $\hat{\gamma}_{i,u}(S')\subset W_2\cap V_i$, with $S'=((U+s\gamma)\cap V_i)\gamma^{-1}$.  Thus, from the definition of $S$ and $S'$ we have that $S\cup S'=V_i$ and $$\Im(\hat{\gamma}_{i,u})\subset (W_1\cup W_2)\cap V_i\subseteq U\cap V_i\ne V_i.$$ Which is not possible by Proposition \ref{prop:5}. 

In a similar way as in Lemma~\ref{lm:1} we can conclude that all the spaces are walls and $\cL\cA_{U}(W_1|W_2)$ is linear.
\end{proof}

\begin{lemm}\label{lm:22}
Let $\gamma=(\gamma_1,...,\gamma_b)$ be a parallel map over $V=V_1\oplus ...\oplus V_b$, $0\gamma=0$, with $V_i=(\FF_2)^m$ and $m\ge 4$. Suppose that for all $i$ $\gamma_i$ is
\begin{enumerate}
\item  differentially $4$-uniform, 
\item  $\hat{n}(\gamma_i)=0$. 
\end{enumerate}
If $\gamma$ maps $\cL\cA_{U}(W_1|W_2)$ onto $\cL(W)$  (non-trivial), then $W$ and $W_1$ are walls and $W=W_1=W_2$, implying also that $\cL\cA_{U}(W_1|W_2)$ is linear.
\end{lemm}
\begin{proof}
Let $I=\{i\,:\, \pi_i(W_1)\ne \{0\}\}\ne\emptyset$. Suppose that there exists $i\in I$ such that $V_i\cap W_1\subsetneq V_i$. If $V_i$ is also such that $V_i\cap U=V_i$, then the conditions on $\gamma_i$ are sufficient to obtain a contradiction. Indeed, following the proof of Lemma~\ref{lm:2}, there exists a nonzero $u\in W_1\cap V_i$ and we will have $\Im(\hat{\gamma}_{i,u})\subset W\cap V_i$. Since $\hat{n}(\gamma_i)=0$, this implies $W\cap V_i=V_i$.
Thus for all these indexes $i$ we have $W\cap V_i= W_1\cap V_i=V_i$.

Suppose that $V_i\cap U\ne V_i$, then $\dim(V_i\cap U)=m-1$.
Following the same steps of the proof of Lemma~\ref{lm:2},  there exists $u\in W_1\cap V_i$ nonzero and we would have that $\hat{\gamma}_{i,u}(S)\subset W\cap V_i$ with $S=(U\cap V_i)$. The set $S$ contains half of the elements of $V_i$. 
As in Lemma~\ref{lm:2}, this implies $\dim(W\cap V_i)\ge m-2$ and thus also $\dim(W_1\cap V_i)\ge m-2$.\\
Now, we can choose $\bar v\in V_i$ such that $\bar v \notin U$ and then, $(W_2+\bar v)\gamma=W+\bar v\gamma$. Since $\gamma$ is a parallel map we have $(W_2\cap V_i+\bar v)\gamma=(W\cap V_i)+\bar v\gamma$. This implies that 
$\dim(W_2\cap V_i)\ge m-2$.
Moreover, for all $u\in W_2\cap V_i$ we have $\hat{\gamma}_{i,u}(S')\subset W\cap V_i$, with $S'=(U+\bar v)\cap V_i$.
Note that $S\cup S'=V_i$. 

So, $m\ge4$ 
implies that $W_1\cap W_2\cap V_i\ne \{0\}$, which means that there exists $u \in W_1\cap W_2\cap V_i$ nonzero.
  Thus, $\Im(\hat{\gamma}_{i,u})\subset W\cap V_i\ne V_i$. Which is not possible since $\hat{n}(\gamma_i)=0$.
  
As in Lemma~\ref{lm:2}, the spaces are walls and $\cL\cA_{U}(W_1|W_2)$ is linear.
\end{proof}

Mixing the proofs of the two lemmas above we can obtain the following result.
\begin{lemm}\label{lm:33}
Let $\gamma=(\gamma_1,...,\gamma_b)$ be a parallel map over $V=V_1\oplus ...\oplus V_b$, $0\gamma=0$, with $V_i=(\FF_2)^m$ and $m\ge 4$. Suppose that for all $i$ $\gamma_i$ is
\begin{enumerate}
\item  differentially $4$-uniform, 
\item  $\hat{n}(\gamma_i)=0$. 
\end{enumerate}

If $\gamma$ maps $\cL\cA_{U}(W_1|W_2)$ onto  $\cL\cA_{U'}(W'_1|W'_2)$ (non-trivial), with $U$ and $U'$ such that $J_U\cap J_{U'}= \emptyset$. Then $W_{1,2}$ and $W'_{1,2}$ are walls and $W_1=W'_1=W_2=W'_2$. In particular, both partitions are linear.
\end{lemm}
\begin{proof}
As before, let $I=\{i\,:\, \pi_i(W_1)\ne \{0\}\}\ne\emptyset$. Suppose that there exists $i\in I$ such that $V_i\cap W_1\subsetneq V_i$. We can have the following cases:
\begin{enumerate}
\item $i \notin J_U\cup J_{U'}$;
\item $i$ in $J_U$ but not in $J_{U'}$;
\item$i$ in $J_{U'}$ but not in $J_{U}$.
\end{enumerate}
\noindent{\bf CASE 1.}\\
We have that $U\cap V_i=U'\cap V_i=V_i$. 
Then, as in Lemma~\ref{lm:3} we can determine a nonzero element $u\in V_i\cap W_1$ and $$\hat{\gamma}_{i,u}(V_i\cap U)=\Im(\hat{\gamma}_{i,u})\subset W_1\cap V_i.$$ Which is not possible since $\hat{n}(\gamma_i)=0$.

\noindent{\bf CASE 2.}\\
We have that $U\cap V_i=V_i$ and $U'\cap V_i\ne V_i$. Then there exists $\bar v'\in V_i$ with  $\bar v'\notin U'$.  We can define the sets $S=(U'\cap V_i)\gamma^{-1}$ and $S'=((U'+\bar v')\cap V_i)\gamma^{-1}$, both containing  $2^{m-1}$ elements and $S\cup S'=V_i$.

Then, as in Lemma~\ref{lm:3}, we can determine a nonzero element $u\in V_i\cap W_1$ with $\hat{\gamma}_{i,u}(S)\subset W'_1\cap V_i$ and $\hat{\gamma}_{i,u}(S')\subset W'_2\cap V_i$. Thus $$\Im(\hat{\gamma}_{i,u})\subset U'\cap V_i\ne V_i,$$ which is not possible since $\hat{n}(\gamma_i)=0$.

\noindent{\bf CASE 3.}\\
We have that $U\cap V_i\ne V_i$ and $U'\cap V_i= V_i$. Then, there exists $\bar v\in V_i$ with  $\bar v\notin U$. Let $S=(U\cap V_i)$ and $S'=((U+\bar v)\cap V_i)$, both containing  $2^{m-1}$ elements and $S\cup S'=V_i$.

As in Lemma~\ref{lm:3} we can determine a nonzero element $u\in V_i\cap W_1$ so that $\hat{\gamma}_{i,u}(S)\subset W'_1\cap V_i$, implying $\dim(W_1\cap V_i)=\dim(W'_1\cap V_i)\ge m-2$.

Moreover, since $\bar v\in V_i$ we have $\bar v\gamma \in U'$ and thus $(W_2+\bar v)\gamma=W'_1+\bar v\gamma$. Since $\gamma$ is a parallel map we have $((W_2\cap V_i)+\bar v)\gamma=(W'_1\cap V_i)+\bar v\gamma$. This implies that 
$\dim(W_2\cap V_i)\ge m-2$.

So, $m\ge4$ 
implies that $W_1\cap W_2\cap V_i\ne \{0\}$, this means that there exists $u \in W_1\cap W_2\cap V_i$ nonzero. For such an element $u$ we obtain 
$$
\Im(\hat{\gamma}_{i,u})\subset W'_1\cap V_i\ne V_i. 
$$
Which is not possible since $\hat{n}(\gamma_i)=0$.

The fact that all the spaces coincide and are walls follows similarly to Lemma~\ref{lm:3}.
\end{proof}

\begin{remar}
In dimension $4$ there exist several permutations which are differentially $4$-uniform and such that $\hat n (f)=0$.  Moreover, for any dimension $m$ the inversion map $\gamma:x\mapsto x^{-1}$ is differentially $4$-uniform if $m$ is even and APN if $m$ is odd, and it is such that $\hat{n}(\gamma)=0$ for any dimension $m$ (see for instance \cite[Corollary 6]{Kyu}).
\end{remar}

\subsection{Primitivity of $\Gamma(\cC)$ in the case of subgroups of order $2^{n-1}$}

Using the results above we can define the following property for a key-schedule of a tb cipher.

\begin{definitio}\label{def:3rai}
Let $\cC$ be a tb cipher over $V$, with key space $\cK$. Let $\ell$ be the number of the rounds and $\Phi:\cK\to V^\ell$ be the key-schedule of $\cC$. Let $T_1,T_2$ and $T_3$ be three subgroups of $T(V)$ of order at least $2^{n-1}$ and $U_j=\{v\,:\,\gs_v\in T_j\}$ for $j=1,2,3$. 

Suppose that there exists $2\le i\le \ell-1 $ such that for any $(k_{i-1},k_i,k_{i+1})\in U_1\times U_2\times U_3$ the element 
$(\bar{k}_1,...,\bar{k}_{i-2},k_{i-1},k_i,k_{i+1},\bar{k}_{i+2},...,\bar{k}_\ell)$ is in $\Im(\Phi)$, where the values $\bar{k}_j$'s are fixed. Then, we say that the key-schedule $\Phi$ is {\em 3-round almost-independent at round $i$} with respect to $T_1,T_2$ and $T_3$.
\end{definitio}

An easy example of a 3-round almost-independent key-schedule is given by the following. Let $f_1,f_2$ and $f_3$ be three fixed linear Boolean functions from $(\FF_{2})^n$ to $\f2$. Then, a map $$
\begin{aligned}\Phi:(\f2)^{4n}&\to(\f2)^{\ell n}\\
(k_1,k_2,k_3,k_4)&\mapsto (k_1,k_2,k_3,\Phi'(f_1(k_1),f_2(k_2),f(k_3),k_4)),\end{aligned}$$ with $\Phi'$ any function from $(\f2)^{3+n}$ to $(\f2)^{(\ell-3)n}$, satisfies the 3-round almost-independent property (the groups $T_1,T_2,T_3$ depend on the functions $f_1,f_2,f_3$).\\

For the mixing layer it is easy to extend Proposition \ref{prop:ml} to the case of linear-affine partitions.
\begin{propositio}\label{prop:ml2}
Let $\gl$ be a linear permutation of $V$, and let $\cL\cA_{U}(W_1|W_2)$ be a linear-affine partition of $V$. Then $\cL\cA_{U}(W_1|W_2)\gl=\cL\cA_{U\gl}(W_1\gl |W_2\gl)$.
\end{propositio}

From the analysis reported above we are able to state the following result. 

\begin{theore}\label{th:mainme}
Let $\cC$ be a tb cipher with a key-schedule function $\Phi$ 3-round almost-independent at round $i$ with respect to $T_1,T_2$ and $T_3$, for some $2\le i\le \ell-1 $ and $T_1,T_2$ and $T_3$ subgroups of $T(V)$ of order at least $2^{n-1}$. Let $U_j=\{v\,:\,\gs_v\in T_j\}$ for $j=1,2,3$. If the parallel S-boxes of round $i$ and round $i+1$, $\gamma_i$ and $\gamma_{i+1}$, are composed by S-boxes which are
\begin{enumerate}
\item  differentially $2^r$-uniform, with $r<m-1$, 
\item strongly $r$-anti-invariant; 
\end{enumerate}
\begin{center}
or
\end{center}
\begin{enumerate}
\item[(i)]  differentially $4$-uniform, 
\item[(ii)]  $\hat{n}(\gamma_{i})=\hat{n}(\gamma_{i+1})=0$, 
\end{enumerate}
$\gl_i$ is strongly proper and $J_{U_2\gl_i^{-1}}\cap J_{U_{1}}=\emptyset$, then there do not exist $\cA$ and $\cB$ non-trivial partitions such that for all $k\in\cK$ the map $\tau_k$ maps $\cA$ onto $\cB$. In particular, $\Gamma(\cC)$ is primitive.
\end{theore}
\begin{proof}
The proof is similar to Theorem~\ref{th:gam}, with the only difference that in this case the partitions $\cA_1,\cA_2$ and $\cA_3$ can be also linear-affine.

Since $\Phi$ is 3-round almost-independent at round $i$ with respect to $T_1,T_2,T_3$, there exist some fixed values $\bar{k}_j$, with $j\in \{1,...,\ell\}\setminus\{i-1,i,i+1\}$, such that 
$(\bar{k}_1,...k_{i-1},k_i,k_{i+1},...,\bar{k}_\ell)$ is in $\Im(\Phi)$, for any $(k_{i-1},k_i,k_{i+1})\in U_1\times U_2\times U_3$.\\
Denote $\gt_j=\gamma_j\gl_j$, and consider the encryption maps 
\begin{equation}\label{eq:gtK2}
\tau_K=\bar{\gt}_1 \gs_{k_{i-1}}\gamma_i\gl_i\gs_{k_i}\gamma_{i+1}\gl_{i+1}\gs_{k_{i+1}}\bar{\gt}_2,
\end{equation}
where  $\bar{\gt}_1=\tau_{1}\gs_{\bar{k}_1}\cdots\tau_{i-2}\gs_{\bar{k}_{i-2}}\tau_{i-1}$, $\bar{\gt}_2=\tau_{i+2}\gs_{\bar{k}_{i+2}}\cdots\tau_{\ell}\gs_{\bar{k}_{\ell}}$, and
$(k_{i-1},k_i,k_{i+1})\in U_1\times U_2\times U_3$.

Suppose that all the encryption functions map $\cA$ onto $\cB$. Denoting by $\cA_1=\cA\bar{\gt}_1$ and $\cB'=\cB\gr^{-1}$ with $\gr=\gt_i\gt_{i+1}\bar\gt_2$, we have that, for all $k_{i-1}\in U_1$, $\gs_{k_{i-1}}$ maps $\cA_1$ into $\cB'$. From Proposition \ref{prop:blocchi2} follows that $\cA_1=\cB'$ and it is linear or linear-affine. Similarly, the partitions $\cA_2=\cA_1\gt_i$ and $\cA_3=\cA_2\gt_{i+1}$ are linear or linear-affine.

Now, suppose that all these partitions are linear-affine (the proof for the other possible cases follows in a similar way). Then $\cA_1=\cL\cA_{U_1}(W_1|W_2)$, $\cA_2=\cL\cA_{U_2}(W'_1|W'_2)$ and $\cA_1=\cL\cA_{U_3}(W''_1|W''_2)$ where $W_{1,2}\subseteq U_1$, $W'_{1,2}\subseteq U_2$ and $W''_{1,2}\subseteq U_3$.

From Proposition \ref{prop:ml2} we have that $\gamma_i$ maps the partition $\cL\cA_{U_1}(W_1|W_2)$ into $\cL\cA_{U_2\gl_i^{-1}}(W'_1\gl_i^{-1}|W'_2\gl_i^{-1})$. Thus, from Lemma~\ref{lm:3} or Lemma~\ref{lm:33} we have that $W_1$ and $W'_1\gl_i^{-1}$ are wall, and in addition the two partitions are linear. Thus, we have that $\cA_2=\cL(W'_1)$ and, now, $\gamma_{i+1}$ maps the partition $\cL(W'_1)$ into the partition $\cL\cA_{U_3\gl_{i+1}^{-1}}(W''_1\gl_{i+1}^{-1}|W''_2\gl_{i+1}^{-1})$. From Lemma~\ref{lm:1} or Lemma~\ref{lm:11}, we have that also $W_1'$ is a wall, contradiction. 
Hence, there do not exist non-trivial partitions $\cA$ and $\cB$.
\end{proof}

We can note that if any of the three groups is the translation group $T(V)$, then one of the partition $\cA_i$ can be only linear. Using this fact, we have the following corollary.

\begin{corollary}\label{cor:nonlimit}
Let $\cC$ be a tb cipher with a key-schedule $\Phi$ 3-round almost-independent at round $i$ with respect to $T_1,T_2$ and $T_3$, for some $2\le i\le \ell-1 $.  Suppose that for one $j$, $T_j=T(V)$. If the parallel S-boxes of round $i$ and round $i+1$ satisfy the conditions of Theorem~\ref{th:mainme} and $\gl_i$ is strongly proper, then $\Gamma(\cC)$ is primitive.
\end{corollary}
\begin{proof}
As in Theorem~\ref{th:mainme}, we have the partitions $\cA_1,\cA_2$ and $\cA_3$. Suppose that $T_1=T(V)$, and then $\cA_1=\cL(W)$ is linear. Thus, we can apply Lemma~\ref{lm:1} or Lemma~\ref{lm:11}. From this, we have that $W$ is a wall and $\cA_2$ is linear with $\cA_2=\cL(W\gl_i)$. Again, we can apply Lemma~\ref{lm:1} or Lemma~\ref{lm:11} obtaining that $W\gl_i$ is a wall.

Now, suppose $T_2=T(V)$ and $\cA_2=\cL(W)$. From Lemma~\ref{lm:2} or Lemma~\ref{lm:22} we have that $\cA_1$ is linear and $W\gl_i^{-1}$ is a wall. Then, applying always Lemma~\ref{lm:1} or Lemma~\ref{lm:11} we can obtain that also $W$ is a wall.

For the last case, we have $\cA_3=\cL(W)$. Then, $\gamma_{i+1}$ maps $\cA_2$ (linear or linear-affine) into $\cL(W\gl_{i+1}^{-1})$. We can apply Lemma~\ref{lm:2} or Lemma~\ref{lm:22} to have that $\cA_2$ is linear with $\cA_2=\cL(W')$ and $W'=W\gl_{i+1}^{-1}$ a wall. Now, using Lemma~\ref{lm:2} or Lemma~\ref{lm:22} we can also obtain that $\cA_1=\cL(W'\gl_i^{-1})$ and also $W'\gl_i^{-1}$ is a wall. In all the cases we obtain a contradiction to the fact that $\gl_i$ is strongly proper.
\end{proof}

\section{Conclusion}
In this paper we have introduced properties of the key-schedule algorithm of a cipher $\cC$ (Definition~\ref{def:key} and Definition~\ref{def:3rai}) that allow to show that the group generated by the encryption functions of the cipher is primitive.
In particular, in Theorem~\ref{th:mainme} with a key-schedule which generates (at least) $2^{3n-3}$ sequences of round keys we provided cryptographic properties for the components of the round functions which guarantee the primitivity. 

For the case of $4$-uniform permutations, in Section~\ref{sec:4uni} we have exhibited a property, related to the linear structures of an S-box, that for dimensions greater than or equal to 4 is much easier to check with respect to anti-invariance.

It would be interesting to further study the properties for a key-schedule in order to obtain similar results with a number of round keys sequences smaller than $2^{3n-3}$. For example, for the well known cipher AES we can have only $2^n$ or $2^{2n}$ sequences.

In \cite{bannier2016partition} (see also \cite{bannier2017partition}) the authors generalized the idea of the primitive trapdoor and used the linear partitions for designing a trapdoored block cipher so that the round functions $\gamma\gl$ maps a linear partition into another one. 

From the study carried out in the present paper, we have also that if we could implement a partition-based trapdoor which works not for all possible the round keys, but just for half of them, then we could use also the linear-affine partitions to design a trapdoored cipher.
However, we have shown that under the properties we have imposed on the components of the cipher, such a scenario cannot occur.

\section*{Acknowledgments}
This research was supported by Trond Mohn Stiftelse (TMS) foundation.


\begin{thebibliography}{99}

\bibitem{weak}  R. Aragona, M. Calderini, D. Maccauro and M. Sala, { On weak differential uniformity of vectorial Boolean functions as a cryptographic criterion}, {\em Appl. Algebra Engrg. Comm. Comput.} {\bf 27}(5) (2016) 359--372.

\bibitem{calderinisalerno} R.~Aragona, M.~Calderini, A.~Tortora and M.~Tota, {Primitivity of {PRESENT} and other lightweight ciphers}, \emph{J. Algebra Appl.} {\bf 17}(6) (2018) 16.

\bibitem{aragona2014} R. Aragona, A. Caranti, F. Dalla Volta and M. Sala, {On the group generated by the round functions of translation based ciphers over arbitrary finite fields}, {\em Finite Fields Appl.} {\bf 25} (2014) 293--305.

\bibitem{bannier2016partition} A.~Bannier, N.~Bodin and E.~Filiol, {Partition-Based Trapdoor Ciphers}, \emph{Cryptology ePrint Archive}, Report 2016/493 (2016).

\bibitem{bannier2017partition} A. Bannier and E. Filiol, \emph{Partition-based trapdoor ciphers} (InTechOpen 2017).

\bibitem{CGC2-cry-art-biham1991differential}
E.~Biham and A.~Shamir, {Differential cryptanalysis of {DES-like} cryptosystems}, \emph{J.~Cryptology} {\bf 4}(1) (1991) 3--72.
\bibitem{dillon}  K. A. Browning, J. F. Dillon, M. T. McQuistan and A. J. Wolfe,{ An APN permutation in dimension six}, in {\em Finite fields: theory and applications--Fq9}, Contemp. Math., vol. 518 (Am. Math Soc., Providence, 2010) pp. 33--42.

\bibitem{calderini2018} M. Calderini, {A note on some algebraic trapdoor}, \emph{Adv. Math. Commun.} {\bf 12}(3) (2018) 515--524.
\bibitem{calapn} M. Calderini, M. Sala and I. Villa, A note on APN permutations in even dimension, \emph{ Finite Fields and Appl.} {\bf 46} (2017) 1--16.
\bibitem{CGC-cry-art-carantisalaImp} A.~Caranti, F.~{Dalla Volta} and M.~Sala, {On some block ciphers and imprimitive groups}, {\em Appl. Algebra Engrg. Comm. Comput.} {\bf 20}(5-6) (2009) 229--350.
\bibitem{caranti}  A. Caranti, F. Dalla Volta and M. Sala, An application of the O'Nan-Scott theorem to the group generated by the round functions of an AES-like cipher, {\em Des. Codes
Cryptogr.} {\bf 52}(3) (2009) 293--301.

  \bibitem{camp} K. W. Campbell and M. J. Wiener. {DES is not a Group}, in {\em Advances in Cryptology--Crypto'92 Proceedings}, Lecture Notes in Comput. Sci., vol. 740 (Springer, Berlin, 1992) pp. 512--520.

\bibitem{carlet} C. Carlet, Vectorial Boolean Functions for Cryptography, in  {\em Boolean Models and Methods in Mathematics, Computer Science, and Engineering}, eds. Y. Crama,  P. L. Hammer (Cambridge University Press, Oxford, 2010) pp. 257--397.
  
  \bibitem{cop} D. Coppersmith, {The Real Reason for Rivest's Phenomenon}, in {\em Advances in Cryprology--Crypto'85 Proceedings}, Lecture Notes in Comput. Sci., vol. 218 (Springer, Berlin, 1995) pp .535--536. 

  \bibitem{coppersmith1975generators} D.~Coppersmith and E.~Grossman, {Generators for certain alternating groups with applications to cryptography}, \emph{SIAM J. Appl. Math.} {\bf 29}(4) (1975) 624--627.




%
  
\bibitem{even} S. Even and O. Goldreich, {DES-Like functions can generate the alternating group}, \emph{ IEEE Trans. Inform. Theory} {\bf 29}(6) (1983) 863--865.

\bibitem{harpes1997partitioning} C.~Harpes and J.~L. Massey, {Partitioning cryptanalysis}, in: \emph{Fast Software Encryption}, Lecture Notes in Comput. Sci., vol. 1267 (Springer, Berlin, 1997) 13--27.
  
\bibitem{hou} X.-D. Hou,{ Affinity of permutations of $\mathbb{F}_2^n$}, {\em Discrete Appl. Math.} {\bf 154}(2) (2006) 313--325.
  
\bibitem{Kaliski} B. S. Kaliski, R. L. Rivest, and A. T. Sherman, {Is the Data Encryption Standard a group?(Results of cycling experiments on DES)}, \emph{J. Cryptology} {\bf 1}(1) (1988) 3--36.

\bibitem{Kyu} G. M. Kyureghyan,{ Crooked maps in $\fn$}, {\em Finite Fields and Appl.} {\bf 13}(3) (2007) 713--726.

\bibitem{CGC-cry-art-paterson1} K.~G. Paterson, {Imprimitive permutation groups and trapdoors in iterated block ciphers}, in: \emph{Fast software encryption}, Lecture Notes in Comput. Sci., vol. 1636 (Springer, Berlin, 1999) 201--214.


\bibitem{MKPW} S. Murphy, K. Paterson, and P. Wild, { A weak cipher that generates the symmetric group}, {\em J. Cryptology} {\bf 7}(1) (1994) 61--65.




  
\bibitem{Wernsdorf2} R. Sparr and R. Wernsdorf, {Group theoretic properties of Rijndael-like ciphers}, \emph{Discrete Appl. Math.} {\bf 156}(16) (2008) 3139--3149.
  
\end{thebibliography}
\end{document}